\documentclass{amsart}
\usepackage{amsmath}
\usepackage{amssymb}
\usepackage{amsfonts}

\newtheorem{theorem}{Theorem}
\theoremstyle{plain}

\newtheorem{corollary}{Corollary}

\newtheorem{definition}{Definition}

\newtheorem{lemma}{Lemma}

\newtheorem{proposition}{Proposition}
\newtheorem{remark}{Remark}

\numberwithin{equation}{section}

\begin{document}
\title{The Differential Brauer Group}
\author{Raymond T. Hoobler}
\curraddr{Department of Mathematics, City College of New York, 10031}
\email{rhoobler@ccny.cuny.edu}

\begin{abstract}
Let $A$ be a ring equipped with a derivation $\delta $. We study
differential Azumaya $A$ algebras, that is, Azumaya $A$ algebras equipped
with a derivation that extends $\delta $. We calculate the differential
automorphism group of the trivial differential algebra, $M_{n}(A$) with
coordinatewise differentiation. We introduce the $\delta$-flat Grothendieck
topology to show that any differential Azumaya $A$ algebra is locally
isomorphic to a trivial one and then construct, as in the non-differential
setting, the embedding of the differential Brauer group into $H^{2}(
A_{\delta-pl},G_{m,\delta }) $. We conclude by showing that the
differential Brauer group coincides with the usual Brauer group in the
affine setting.
\end{abstract}
\maketitle

We wish to complete and extend the study of differential central separable
algebras that was begun by Juan and Magid \cite{MR2383496}. First we want to
extend their material from differential fields to differential rings,
proceeding in the obvious manner, by defining differential Azumaya algebras.
Then we define the differential Brauer group, $Br_{\delta },$ by using the
usual definition and applying it to the category of differential projective
modules. Note that this does not agree with the definition in \cite%
{MR2383496} where the equivalence relation uses only vector spaces over the
differential field, not differential vector spaces. Finally we will give a
cohomological characterization of $Br_{\delta }$ which can be easily
compared with the cohomological characterization of the usual Brauer group
of a ring.

In order to carry out this program, we will prove a differential Morita
theorem and use it to identify differential automorphisms of a differential
Azumaya algebra. Then we introduce a new Grothendieck topology, the $\delta
-flat,$ topology which allows us to solve differential equations locally and
then we use hypercoverings to construct a differential boundary map for a
central extension of sheaves. This allows us to embed $Br_{\delta }\left(
A\right) $ into $H^{2}\left( A_{\delta -pl},\mathbb{G}_{m,\delta }\right) .$
This part is a straightforward extension to the differential category of the
basic results for the Brauer group of a commutative ring. Our main result,
surprisingly enough, is $Br_{\delta }\left( A\right) =Br\left( A\right) $
for a differential ring $A.$

We fix a commutative differential ring $\left( A,\delta \right) $ throughout
which, for simplicity, we assume to be connected for the main result. Tensor
products are always over $A$ unless otherwise indicated. We use $\delta $ to
indicate the derivation and leave the reader to identify where it is acting.
In a similar vein we use a subscript $\delta $ to indicate that we are
operating in a differential setting and a superscript $\delta $ to indicate
that we are considering differential invariants. We require that all
differential rings contain $\mathbb{Q}$. What happens for more general
commutative differential rings, particularly in the unequal characteristic
case, remains an interesting open question.

\section{Differential Brauer Group}

Let $\Pr oj_{\delta }\left( A\right) $ denotes the set of differential $A$
modules that are finitely generated, faithful, projective $A$ modules, $%
Az_{\delta }\left( A\right) $ denotes the set of pairs $\left( \Lambda
,\delta \right) $ such that $\Lambda $ is an Azumaya $A$ algebra equipped
with a derivation $\delta $ making $\Lambda $ into a differential $A$
algebra. Let $Az_{n,\delta }\left( A\right) \subset Az_{\delta }\left(
A\right) $ be those differential Azumaya algebras $\Lambda $ with $%
rank\left( \Lambda \right) =n^{2},$ and similarly let $\Pr oj_{n,\delta
}\left( A\right) \subset \Pr oj_{\delta }\left( A\right) $ be those finitely
generated projective differential $A$ modules $P$ with $rank\left( P\right)
=n.$

If $P,Q\in \Pr oj_{\delta }\left( A\right) ,$ $P\otimes Q\in \Pr oj_{\delta
}\left( A\right) $ by defining $\delta \left( p\otimes q\right) :=\delta
p\otimes q+p\otimes \delta q.$ The $A$ dual, $P^{\vee }=Hom_{A}\left(
P,A\right) ,$ also has a derivation defined by setting $\left( \delta
f\right) \left( p\right) =\delta \left( f(p)\right) -f\left( \delta p\right)
.$ The isomorphism
\begin{equation*}
Q\otimes P^{\vee }\overset{\alpha }{\rightarrow }Hom_{A}\left( P,Q\right)
\end{equation*}
given by $\alpha \left( q\otimes f\right) \left( p\right) =qf\left( p\right)
$ makes $Hom_{A}\left( P,Q\right) \in \Pr oj_{\delta }\left( A\right) $ for $%
P,Q\in \Pr oj_{\delta }\left( A\right) $ by defining $\delta \left( \alpha
\left( q\otimes f\right) \right) =\alpha \left( q\otimes \left( \delta
f\right) +\left( \delta q\right) \otimes f\right) .$ This leads to the
formula%
\begin{equation*}
\delta \left( \alpha \left( q\otimes f\right) \right) \left( p\right)
=q\left( \delta \left( f(p)\right) -f\left( \delta p\right) \right) +\left(
\delta q\right) f(p)
\end{equation*}%
that makes $Hom_{A}\left( P,Q\right) $ into a differential $A$ module. We
introduce the notation $Hom_{A}\left( P,Q\right) _{\delta }$ to indicate the
module $Hom_{A}\left( P,Q\right) $ equipped with its derivation. Note that
if $P,Q,P^{\prime },Q^{\prime }\in \Pr oj_{\delta }\left( A\right) ,$ then a
straightforward but tedious calculation shows that the $A$ module isomorphism%
\begin{equation*}
\phi :Hom_{A}\left( P,Q\right) \otimes Hom_{A}\left( P^{\prime },Q^{\prime
}\right) \rightarrow Hom_{A}\left( P\otimes P^{\prime },Q\otimes Q^{\prime
}\right)
\end{equation*}%
defined by $\phi \left( f\otimes f^{\prime }\right) \left( p\otimes
p^{\prime }\right) =f\left( p\right) \otimes f^{\prime }\left( p^{\prime
}\right) $ is a differential $A$ module isomorphism.

Once this framework is set up, we can translate the definitions required to
define $Br\left( A\right) $ into the differential setting.

\begin{definition}
Let $\Lambda ,\Gamma \in Az_{\delta }\left( A\right) .$ $\Lambda \backsim
_{\delta }\Gamma $ if there are $P,Q\in \Pr oj_{\delta }\left( A\right) $
and a differential $A$ algebra isomorphism
\begin{equation*}
\Lambda \otimes End_{A}\left( P\right) \approxeq _{\delta }\Gamma \otimes
End_{A}\left( Q\right) .
\end{equation*}
\end{definition}

\begin{lemma}
$\backsim _{\delta }$ is an equivalence relation.
\end{lemma}

\begin{proof}
Clear once the above tedious calculation is done.
\end{proof}

We introduce the notation $\left[ \Lambda \right] $ for the equivalence
class of $\Lambda $ and use the same notation in the non-differential case.

\begin{definition}
$Br_{\delta }\left( A\right) =\left\{ \left[ \Lambda \right] \mid \Lambda
\in Az_{\delta }\left( A\right) \right\} $
\end{definition}

\begin{lemma}
\begin{enumerate}
\item If $\Lambda ,\Gamma ,\Lambda ^{\prime },\Gamma ^{\prime }\in
Az_{\delta }\left( A\right) ,$ $\Lambda \backsim _{\delta }\Lambda ^{\prime
} $, and $\Gamma \backsim _{\delta }\Gamma ^{\prime },$ then $\Lambda
\otimes \Gamma \backsim _{\delta }\Lambda ^{\prime }\otimes \Gamma ^{\prime
}.$

\item If $\Lambda \in Az_{\delta }\left( A\right) ,$ then the algebra
isomorphism $\rho :\Lambda \otimes \Lambda ^{op}\rightarrow End_{A}\left(
\Lambda \right) $ given by $\rho \left( \lambda _{1}\otimes \lambda
_{2}\right) \left( \lambda \right) =\lambda _{1}\lambda \lambda _{2}$ is a
differential isomorphism.
\end{enumerate}
\end{lemma}

\begin{proof}
The first assertion is simply a rearrangement of factors in the tensor
product and so is immediately verified. The second statement requires a
rather nasty, but typical, calculation, but first we must identify how $\rho
$ acts in terms of $\Lambda \otimes \Lambda ^{\vee }.$ Since $\alpha $ is an
isomorphism, there are elements $\lambda _{\beta }\in \Lambda $ and $%
f_{\beta }\in \Lambda ^{\vee }$ such that
\begin{equation*}
\alpha \left( \sum \lambda _{\beta }\otimes f_{\beta }\right) =1_{End\left(
\Lambda \right) }
\end{equation*}%
and so
\begin{equation*}
\rho \left( \lambda _{1}\otimes \lambda _{2}^{op}\right) =\alpha \left( \sum
\lambda _{1}\lambda _{\beta }\lambda _{2}\otimes f_{\beta }\right) .
\end{equation*}%
Now%
\begin{eqnarray*}
\rho \left( \delta \lambda _{1}\otimes \lambda _{2}^{op}+\lambda _{1}\otimes
\left( \delta \lambda _{2}\right) ^{op}\right) \left( \lambda \right)
&=&\left( \delta \lambda _{1}\right) \lambda \lambda _{2}+\lambda
_{1}\lambda \left( \delta \lambda _{2}\right) \text{ and} \\
\left( \delta \rho \left( \lambda _{1}\otimes \lambda _{2}^{op}\right)
\right) \left( \lambda \right)  &=&\alpha \left( \sum \left( \delta \left(
\lambda _{1}\lambda _{\beta }\lambda _{2}\right) \right) \otimes f_{\beta
}+\lambda _{1}\lambda _{\beta }\lambda _{2}\otimes \left( \delta f_{\beta
}\right) \right) \left( \lambda \right)  \\
&=&\sum \left( \delta \lambda _{1}\right) \lambda _{\beta }\lambda
_{2}f_{\beta }\left( \lambda \right) +\sum \lambda _{1}\left( \delta \lambda
_{\beta }\right) \lambda _{2}f_{\beta }\left( \lambda \right)  \\
&&+\sum \lambda _{1}\lambda _{\beta }\left( \delta \lambda _{2}\right)
f_{\beta }\left( \lambda \right) +\sum \lambda _{1}\lambda _{\beta }\lambda
_{2}\left( \delta f_{\beta }\right) \left( \lambda \right)  \\
&=&\left( \delta \lambda _{1}\right) \lambda \lambda _{2}+\lambda
_{1}\lambda \left( \delta \lambda _{2}\right) +\lambda _{1}\left[ \sum
\left( \delta \lambda _{\beta }\right) f_{\beta }\left( \lambda \right)
+\lambda _{\beta }\left( \delta f_{\beta }\right) \left( \lambda \right) %
\right] \lambda _{2}.
\end{eqnarray*}%
But $\alpha \left( \sum \delta \left( \lambda _{\beta }\otimes f_{\beta
}\right) \right) =0=\alpha \left( \sum \left( \delta \lambda _{\beta
}\right) \otimes f_{\beta }+\lambda _{\beta }\otimes \left( \delta f_{\beta
}\right) \right) $ and so
\begin{equation*}
\lambda _{1}\left[ \sum \left( \delta \lambda _{\beta }\right) f_{\beta
}\left( \lambda \right) +\lambda _{\beta }\left( \delta f_{\beta }\right)
\left( \lambda \right) \right] \lambda _{2}=0
\end{equation*}%
also.
\end{proof}

\begin{corollary}
$Br_{\delta }\left( A\right) $ is an abelian group if multiplication is
defined as $\left[ \Lambda \right] \cdot \left[ \Gamma \right] =\left[
\Lambda \otimes \Gamma \right] $ and $\left[ \Lambda \right] ^{-1}=\left[
\Lambda ^{op}\right] .$
\end{corollary}

Note that there is an obvious map
\begin{equation*}
\iota :Br_{\delta }\left( A\right) \rightarrow Br\left( A\right)
\end{equation*}%
given by forgetting the differential structure on a differential Azumaya
algebra.

\section{Differential Morita Theory}

Most information about the Brauer group of a ring is obtained by identifying
it with a second cohomology group. This requires some descent theory and an
understanding of automorphisms of an Azumaya algebra. The best way of
developing this understanding comes from the Morita theorems. In this
section we extend this theory to the differential setting by showing that
the isomorphisms involved are differential isomorphisms in our setting.

We adopt the notation in (\cite[Chapter 1, 9.1]{MR1096299}) which means
altering our notation slightly in this section. We fix a differential
commutative ring $R$ and assume that $A$ is a not necessarily commutative
differential $R-$algebra. Let $P$ be a differential right $A$ module, and
let $B=End_{A}\left( P\right) _{\delta }$ be the differential endomorphism
ring of $P.$ The action of $B$ on $P$ defines the structure of a $B-Mod-A$
bimodule on $P.$ The $A$ dual, $Q=Hom_{A}\left( P,A\right) _{\delta },$ is
then an $A-Mod-B$ bimodule. The module structures have been set up to show
actions via associativity. Furthermore letters $p$ and $q$ with adornments
always indicate elements of $P$ or $Q$ and $\left( qp\right) $ denotes the
value of the functional $q$ on $p.$ Thus associativity of parentheses
indicate the various actions. For instance, $Q$ is an $A-B$ bimodule via $%
\left( aqb\right) \left( p\right) =a\cdot q\left( bp\right) $. \ $P$ and $Q$
are related by two maps%
\begin{eqnarray*}
f_{P} &:&P\otimes _{A}Q\rightarrow B\text{ where }f_{P}\left( p\otimes
q\right) \left( x\right) =p\cdot \left( q\left( x\right) \right) \text{ and}
\\
g_{P} &:&Q\otimes _{B}P\rightarrow A\text{ where }g_{P}\left( q\otimes
p\right) =q\left( p\right) .
\end{eqnarray*}%
$f_{P}$ is a homomorphism of $B$ bimodules, and $g_{P}$ is a homomorphism of
$A$ bimodules. They are associative in the sense that if we write $%
f_{P}\left( p\otimes q\right) =pq$ and $g_{p}\left( q\otimes p\right) =qp,$
then $\left( pq\right) \cdot p^{\prime }=p\cdot \left( qp^{\prime }\right) $
and $\left( qp\right) \cdot q^{\prime }=q\cdot \left( pq^{\prime }\right) .$
In the following $\cdot $ is used to indicate action by an element of $A$ or
$B$ on $P$ or $Q.$

We recall that $P$ is a finitely generated projective $A$ module if and only
if $f_{P}$ is an isomorphism, and $P$ is a faithful projective $A$ module if
and only if $g_{P}$ is an isomorphism (\cite{MR1096299}).

Assume that $P$ is a finitely generated faithful differential projective
right $A$ module. Thus $Q$ becomes a differential left $A$ module by
defining $\left( \delta q\right) \left( p\right) =\delta \left( qp\right)
-q\left( \delta p\right) ,$ $P\otimes _{A}Q$ becomes a differential left $B$
module by defining $\delta \left( p\otimes q\right) :=\delta p\otimes
q+p\otimes \delta q$ or, in the shorthand notation, $\delta \left( pq\right)
=\left( \delta p\right) q+p\left( \delta q\right) ,$ and we also have the
useful identity $\delta \left( qp\right) =\left( \delta q\right) p+q\left(
\delta p\right) $ (Calculate!).

The identities%
\begin{eqnarray*}
\delta \left( pq\right) &=&\left( \delta p\right) q+p\left( \delta q\right)
\text{ and }\delta \left( qp\right) =\left( \delta q\right) p+q\left( \delta
p\right) \text{ together with} \\
\left( pq\right) \cdot p^{\prime } &=&p\cdot \left( qp^{\prime }\right)
\text{ and }\left( qp\right) \cdot q^{\prime }=q\cdot \left( pq^{\prime
}\right)
\end{eqnarray*}%
make differentiation computations easy. Thus we make $B$ into a differential
ring by defining
\begin{equation*}
\delta \left( pq\right) \left( p^{\prime }\right) =\left( \delta p\right)
\cdot \left( qp^{\prime }\right) +p\cdot \left( \left( \delta q\right)
\left( p^{\prime }\right) \right) =\left( \delta p\right) \cdot \left(
qp^{\prime }\right) +p\cdot \delta \left( qp^{\prime }\right) -p\cdot \left(
q\left( \delta p^{\prime }\right) \right) .
\end{equation*}%
In order to verify that this is a derivation of $B,$ we observe that with
the identification using $g_{P},$ we have $\left( pq\right) \left( p^{\prime
}q^{\prime }\right) =p\left( qp^{\prime }\right) \cdot q^{\prime }=p\cdot
\left( qp^{\prime }\right) q^{\prime }$ since $\left( pq\right) \left(
p^{\prime }q^{\prime }\right) \left( p^{\prime \prime }\right) =\left(
pq\right) \left( p^{\prime }\cdot \left( q^{\prime }p^{\prime \prime
}\right) \right) =p\cdot \left( \left( qp^{\prime }\cdot \left( q^{\prime
}p^{\prime \prime }\right) \right) \right) =\left( p\cdot \left( qp^{\prime
}\right) \right) \cdot \left( q^{\prime }p^{\prime \prime }\right) .$ This
makes checking that $B$ is a differential ring straightforward.%
\begin{eqnarray*}
\delta \left( \left( pq\right) \left( p^{\prime }q^{\prime }\right) \right)
\left( p^{\prime \prime }\right) &=&\delta \left( \left( p\cdot \left(
qp^{\prime }\right) \right) q^{\prime }\right) \left( p^{\prime \prime
}\right) =\delta \left( p\cdot \left( qp^{\prime }\right) \right) \cdot
\left( q^{\prime }p^{\prime \prime }\right) + \\
&&\left( p\cdot \left( qp^{\prime }\right) \right) \cdot \delta \left(
q^{\prime }p^{\prime \prime }\right) -p\cdot \left( qp^{\prime }\right)
\cdot \left( q^{\prime }\delta p^{\prime \prime }\right) \\
&=&\left( \delta p\right) \cdot \left( qp^{\prime }\right) \left( q^{\prime
}p^{\prime \prime }\right) +p\left( \delta \left( qp^{\prime }\right)
\right) \left( q^{\prime }p^{\prime \prime }\right) +\left( p\cdot \left(
qp^{\prime }\right) \right) \left[ \left( \delta q^{\prime }\right)
p^{\prime \prime }+q^{\prime }\left( \delta p^{\prime \prime }\right) \right]
\\
&&-p\cdot \left( qp^{\prime }\right) \cdot \left( q^{\prime }\delta
p^{\prime \prime }\right) \\
&=&\left( \delta p\right) \cdot \left( qp^{\prime }\right) \left( q^{\prime
}p^{\prime \prime }\right) +p\left( \delta \left( qp^{\prime }\right)
\right) \cdot \left( q^{\prime }p^{\prime \prime }\right) +\left( p\cdot
\left( qp^{\prime }\right) \right) \left( \delta q^{\prime }\right)
p^{\prime \prime } \\
&=&\left( \delta p\right) \cdot \left( qp^{\prime }\right) \left( q^{\prime
}p^{\prime \prime }\right) +p\left( \delta \left( q\right) p^{\prime
}+q\delta \left( p^{\prime }\right) \right) \cdot \left( q^{\prime
}p^{\prime \prime }\right) +\left( p\cdot \left( qp^{\prime }\right) \right)
\left( \delta q^{\prime }\right) p^{\prime \prime }
\end{eqnarray*}%
while%
\begin{eqnarray*}
\left( \delta \left( pq\right) \left( p^{\prime }q^{\prime }\right) +\left(
pq\right) \delta \left( p^{\prime }q^{\prime }\right) \right) \left(
p^{\prime \prime }\right) &=&\delta \left( pq\right) \left( p^{\prime }\cdot
q^{\prime }p^{\prime \prime }\right) +\left( pq\right) \left[ \left( \delta
p^{\prime }\right) \cdot \left( q^{\prime }p^{\prime \prime }\right)
+p^{\prime }\cdot \left( \left( \delta q^{\prime }\right) \left( p^{\prime
\prime }\right) \right) \right] \\
&=&\left( \delta p\right) \cdot \left( q\left( p^{\prime }\cdot q^{\prime
}p^{\prime \prime }\right) \right) +p\cdot \left( \left( \delta q\right)
\left( p^{\prime }\cdot q^{\prime }p^{\prime \prime }\right) \right) + \\
&&\left( pq\right) \cdot \left( \left( \delta p^{\prime }\right) \cdot
\left( q^{\prime }p^{\prime \prime }\right) \right) +\left( pq\right) \cdot
p^{\prime }\cdot \left( \left( \delta q^{\prime }\right) \left( p^{\prime
\prime }\right) \right) .
\end{eqnarray*}%
The associative identities allow us to finish the argument.

Moreover similar identities show that both $f_{P}$ and $g_{P}$ become
differential bimodule isomorphisms. This requires checking that the
differentiation with respect to $B$ satisfy the Leibniz rule. This follows
by using the associative rules which allow us to switch sides and the $A$
derivations for $P$ and $Q.$ Thus, for $P,$%
\begin{eqnarray*}
\delta \left( \left( pq\right) \cdot p^{\prime }\right) &=&\delta \left(
p\cdot \left( qp^{\prime }\right) \right) =\left( \delta p\right) \cdot
\left( qp^{\prime }\right) +p\cdot \left( \left[ \left( \delta q\right)
p^{\prime }+q\left( \delta p^{\prime }\right) \right] \right) \\
&=&\left( \left( \delta p\right) q\right) \cdot p^{\prime }+\left( p\left(
\delta q\right) \right) \cdot p^{\prime }+\left( pq\right) \cdot \left(
\delta p^{\prime }\right) =\delta \left( pq\right) \cdot p^{\prime }+\left(
pq\right) \cdot \delta p^{\prime }
\end{eqnarray*}%
and, for $Q,$%
\begin{eqnarray*}
\delta \left( q\cdot \left( p^{\prime }q^{\prime }\right) \right) &=&\delta
\left( \left( qp^{\prime }\right) \cdot q^{\prime }\right) =\left[ \left(
\delta \left( q\right) p^{\prime }\right) +\left( q\delta p^{\prime }\right) %
\right] \cdot q^{\prime }+\left( qp^{\prime }\right) \cdot \delta q^{\prime }
\\
&=&\delta q\cdot \left( p^{\prime }q^{\prime }\right) +q\cdot \delta \left(
p^{\prime }q^{\prime }\right) .
\end{eqnarray*}

Once we have verified these basic operations, we know from the standard
Morita theory that \ the functors below are isomorphisms between categories.
Since the maps are differential maps, we conclude that the functors are
isomorphisms of categories of differential modules. We summarize this as a
differential Morita equivalence.

\begin{theorem}[Differential Morita Theorem]
Let $A$ be a ring with a derivation $\delta $. Let $P$ be a faithful
projective differential right $A$ module, $Q=Hom_{A}\left( P,A\right)
_{\delta },$ and $B=End_{A}\left( P\right) _{\delta }.$ Then

\begin{enumerate}
\item The functors
\begin{eqnarray*}
P\otimes _{A}- &:&\left( A,\delta \right) -Mod\rightarrow \left( B,\delta
\right) -Mod \\
-\otimes _{B}P &:&Mod-\left( B,\delta \right) \rightarrow Mod-\left(
A,\delta \right) \\
-\otimes _{A}Q &:&Mod-\left( A,\delta \right) \rightarrow Mod-\left(
B,\delta \right) \\
Q\otimes _{B}- &:&\left( B,\delta \right) -Mod\rightarrow \left( A,\delta
\right) -Mod
\end{eqnarray*}%
are equivalences of categories. They remain equivalences of categories when
restricted to the subcategories of differential projective, resp.
differential faithful projective modules.

\item $P$ and $Q$ are faithful and projective as differentiable $A$ and $B$
modules.

\item The maps $f_{P}$ and $g_{P}$ induce isomorphisms of differentiable
bimodules%
\begin{eqnarray*}
P &\approxeq &Hom_{A}\left( Q,A\right) _{\delta },\text{ }Q\approxeq
Hom_{A}\left( P,A\right) _{\delta } \\
P &\approxeq &Hom_{B}\left( Q,B\right) _{\delta },\text{ }Q\approxeq
Hom_{B}\left( P,B\right) _{\delta }.
\end{eqnarray*}

\item The functor $Hom_{A}\left( P,-\right) _{\delta }:Mod-\left( A,\delta
\right) \rightarrow \left( B,\delta \right) -Mod$ is an equivalence of
categories.
\end{enumerate}
\end{theorem}

\begin{proof}
Since all maps are differentiable in our setting, see (\cite[I, Theorem 9.1.3%
]{MR1096299}) for the proof for rings.
\end{proof}

We use this theorem to generalize the Skolem-Noether theorem to
differentiable rings. We state it in both the commutative ring and the
differential commutative ring case, but the proof in the commutative case is
essentially identical to the differential commutative case and so is
omitted.

\begin{theorem}[Differential Skolem-Noether Theorem]
\begin{enumerate}
\item Let $A$ be a commutative ring , $P$ be a faithful projective $A$
module of rank $n$, and let $B=End_{A}\left( P\right) .$ If $\phi
:End_{A}\left( P\right) \rightarrow End_{A}\left( P\right) $ is an
automorphism, let $P_{\phi }$ denote the module $P$ with $B$ action via $%
\phi ,$ i.e. $b\cdot p=\phi \left( b\right) p.$ Then $L=Hom_{B}\left(
P,P_{\phi }\right) $ is a rank $1$ projective $A$ module of order $n$ in $%
Pic\left( A\right) $ which is, as an $A$ module, isomorphic to $A$ if $P$ is
free. If $L$ is isomorphic to $A,$ then there is a $u\in Aut\left( P\right) $
such that $\phi \left( b\right) =ubu^{-1}.$

\item Let $A$ be a commutative ring with differentiation $\delta $ , $P$ be
a faithful projective differentiable $A$ module of rank $n$, and let $%
B=End_{A}\left( P\right) _{\delta }.$ If $\phi :End_{A}\left( P\right)
_{\delta }\rightarrow End_{A}\left( P\right) _{\delta }$ is a differentiable
automorphism, let $P_{\phi ,\delta }$ denote the differential module $P$
with $B$ action via $\phi ,$ i.e. $b\cdot p=\phi \left( b\right) p.$ Then $%
L=Hom_{B}\left( P,P_{\phi }\right) _{\delta }$ is a rank $1$ projective
differential $A$ module of order $n$ in $Pic_{\delta }\left( A\right) $
which is, as an $A$ module, isomorphic to $A$ if $P$ is free. If $L$ is
isomorphic to $A,$ then there is a $u\in Aut_{A}\left( P\right) $ such that $%
\phi \left( b\right) =ubu^{-1}.$ If $L$ is differentially isomorphic to $A,$
then there is a $u\in Aut_{A,\delta }\left( P\right) $ such that $\phi
\left( b\right) =ubu^{-1}.$
\end{enumerate}
\end{theorem}

\begin{proof}
The proof of the first part is identical to that of the second part except
that there is no differentiability requirement. Hence the argument will be
omitted.

According to the Differential Morita Theorem, we have, by composition, a
differentiable isomorphism of $B$ modules%
\begin{equation*}
P\otimes _{A}\left( Q\otimes _{B}P_{\phi }\right) \rightarrow P_{\phi }.
\end{equation*}%
By the third statement of the Differential Morita Theorem, $Hom_{B}\left(
P,B\right) _{\delta }\otimes _{B}P_{\phi }$ is differentially isomorphic to $%
Q\otimes _{B}P_{\phi }.$ But $P$ is a projective differential $B$ module,
and hence we can conclude that the natural map
\begin{equation*}
\alpha :Hom_{B}\left( P,B\right) _{\delta }\otimes _{B}P_{\phi }\rightarrow
Hom_{B}\left( P,P_{\phi }\right) _{\delta }
\end{equation*}
given by $\alpha \left( f\otimes p\right) \left( p^{\prime }\right) =f\left(
p^{\prime }\right) p$ is an isomorphism of differential $A$ modules. $%
P_{\phi }$ is also a projective $B$ module. This can be seen by localizing $A
$ so that $P=\oplus _{1}^{n}Ae_{i}$ with respect to a basis $\left\{
e_{i}\right\} .$ Then
\begin{equation*}
B=\left( \oplus _{1}^{n}Ae_{i}\right) \otimes _{A}\left( \oplus
_{1}^{n}Ae_{i}^{\vee }\right)
\end{equation*}%
and so $B=\oplus _{1}^{n}P$ is a direct sum of left ideals each isomorphic
to $P.$ Note that $P$ is even a differentiable summand if $\left\{
e_{i}\right\} $ is a basis consisting of differentially constant elements
since then $\delta \left( e_{i}^{\vee }\right) =0.$ \ Since $\phi
:B\rightarrow B$ is a differential isomorphism, we also have $B$ decomposing
into a direct sum of left ideals each isomorphic to $P_{\phi ,\delta }$.
Hence $L$ is a projective differential $A$ module, and so $P\otimes
_{A}L\rightarrow P_{\phi }$ is a differential isomorphism of left $B$
modules.

Suppose $L$ is differentially isomorphic to $A.$ Then we have an isomorphism
of differential $B$ modules $u:P\rightarrow P_{\phi }$ which is in $%
Aut_{A,\delta }\left( P\right) $. Consequently, for any $b\in B,$ we have a
commutative diagram of differential maps
\begin{equation*}
\begin{array}{ccc}
P & \overset{u}{\rightarrow } & P_{\phi } \\
\downarrow b &  & \downarrow \phi \left( b\right)  \\
P & \overset{u}{\rightarrow } & P_{\phi }%
\end{array}%
,
\end{equation*}%
or $\phi \left( b\right) =ubu^{-1}$ as claimed.

Finally since $P\otimes _{A}L\approxeq P$ as an $A$ module, we find that $%
\wedge ^{n}\left( P\otimes _{A}L\right) =\left( \wedge ^{n}P\right) \otimes
_{A}L^{\otimes n}\approxeq \wedge ^{n}P\in Pic\left( A\right) .$
Consequently $L\in Pic\left( A\right) $ is an $n$ torsion, rank one
projective $A$ module.
\end{proof}

\begin{corollary}
\label{der}Let $A$ be a commutative ring and let $\delta $ be a derivation
of $End\left( P\right) .$ Then there is an element $z\in End\left( P\right) $
such that $\delta =\left[ z,-\right] .$
\end{corollary}

\begin{proof}
We sketch the proof. For more details, see \cite[Theorem 4.1]{MR0222073}.
Let $A\left[ \varepsilon \right] $ be the ring of dual numbers defined by $%
\varepsilon ^{2}=0.$ Then an $A$ derivation $\delta $ on $\Lambda $ defines
an $A\left[ \varepsilon \right] $ algebra isomorphism $\Phi _{\delta
}:End_{A}\left( P\right) \otimes A\left[ e\right] \rightarrow End_{A}\left(
P\right) \otimes A\left[ e\right] $ by the rule $\Phi _{\delta }\left(
\lambda +\lambda ^{\prime }\varepsilon \right) =\lambda +\left( \delta
\left( \lambda \right) +\lambda ^{\prime }\right) \varepsilon .$ $\Phi
_{\delta }$ is then an automorphism such that $\Phi _{\delta }\otimes _{A%
\left[ \varepsilon \right] }A=1$ Consequently $L=Hom_{A\left[ \varepsilon %
\right] }\left( P,P_{\Phi _{\delta }}\right) \in Ker\left[ Pic\left( A\left[
\varepsilon \right] \rightarrow Pic\left( A\right) \right) \right] =0$ by
the nilpotent version of  Nakayama's lemma. Thus $\Phi _{\delta }\left(
\lambda \right) =u\lambda u^{-1}$ for some unit $u=u_{0}+u_{1}\varepsilon
\in End_{A}\left( P\right) \otimes A\left[ \varepsilon \right] .$ Since $%
\Phi _{\delta }\otimes _{A\left[ \varepsilon \right] }A=1,$ $u_{0}$ is a
unit in the center $A$ \ So we may assume $u_{0}=1$ by multiplying $u$ by $%
u_{0}^{-1}.$ But then, if $u=1+z\varepsilon ,$
\begin{eqnarray*}
\left( 1+z\varepsilon \right) \left( \lambda +\lambda ^{\prime }\varepsilon
\right) \left( 1-z\varepsilon \right)  &=&\left( \lambda +\left( z\lambda
+\lambda ^{\prime }\right) \varepsilon \right) \left( 1-z\varepsilon \right)
=\lambda +\left( z\lambda +\lambda ^{\prime }-\lambda z\right) \varepsilon
\\
&=&\lambda +\left( \delta \left( \lambda \right) +\lambda ^{\prime }\right)
\varepsilon
\end{eqnarray*}%
and so $\delta \left( \lambda \right) =\left[ z,\lambda \right] $ as claimed.
\end{proof}

Finally we need a criterion distinguishing between the two possibilities for
$L$ in the Differential Skolem Noether Theorem.

\begin{lemma}
\label{mapd}Let $P$ be a differential projective $A$ module, and let $\phi
_{u}:End_{A}\left( P\right) \rightarrow End_{A}\left( P\right) $ be an
algebra homomorphism such that $\phi _{u}\left( x\right) =uxu^{-1}$ for all $%
x\in End_{A}\left( P\right) \ $where $u$ is a unit in $End_{A}\left(
P\right) .$ Then $\phi _{u}$ is a differential homomorphism if and only if $%
u^{-1}\left( \delta u\right) \in A.$
\end{lemma}

\begin{proof}
\begin{eqnarray*}
\delta \left( uxu^{-1}\right) -u\left( \delta x\right) u^{-1} &=&\left(
\delta u\right) xu^{-1}+ux\left( \delta u^{-1}\right) \\
&=&\left( \delta u\right) xu^{-1}-uxu^{-1}\left( \delta u\right) u^{-1} \\
&=&0
\end{eqnarray*}%
if and only if $u^{-1}\left( \delta u\right) x=xu^{-1}\left( \delta u\right)
$ for all $x\in End_{A}\left( P\right) .$
\end{proof}

\section{Cohomological Interlude}

The cohomological interpretation of $Br_{\delta }\left( A\right) $ depends
on descent theory and non-abelian cohomology. The usual reference for the
non-abelian cohomology is Giraud's work which introduces gerbes and other
associated categorical machinery. However there is an alternative using
hypercoverings that we will sketch in this section since it provides the
boundary map associated to a central extension of sheaves of groups that is
necessary to establish the embedding $Br\left( X\right) \hookrightarrow
H^{2}\left( X,\mathbb{G}_{m}\right) .$ This approach is certainly known to
those working with non-abelian sheaf cohomology, but as far as we know, it
has not appeared in print.

We begin with a brief discussion of descent theory that leads to the Cech
non-abelian $H^{1}.$ By now there are several references such as \cite[I.2]%
{MR559531}, \cite[III]{MR1096299}, and \cite[Part 1, 4.2.1]{MR2222646} for
descent theory. We will follow the last reference as it is the most recent
and is quite clear. Basically we copy the material just as we did for the
Differential Morita Theorem and then verify that the derivation aspect is
preserved.

If $A$ is a differential commutative ring, $Mod_{A,\delta }$ denotes the
category of $\ $differential $A$ modules. We fix a faithfully flat
differential ring homomorphism, $A\rightarrow B,$ and define the descent
category $Mod_{A\rightarrow B,\delta }$ as follows:

\begin{eqnarray*}
Ob\left( Mod_{A\rightarrow B,\delta }\right)  &=&\left\{ \left( N,\phi
\right) \mid
\begin{array}{c}
N\in Mod_{B,\delta }\text{ and }\phi :N\otimes B\rightarrow B\otimes N\text{
is a differential} \\
\text{ }B\otimes B\text{ isomorphism satisfying the cocycle condition}%
\end{array}%
\right\}  \\
Mor\left( Mod_{A\rightarrow B,\delta }\right)  &=&\left\{ \alpha :\left(
N,\phi \right) \rightarrow \left( N^{\prime },\phi ^{\prime }\right) \mid
\begin{array}{c}
\alpha :N\rightarrow N^{\prime }\in Mod_{B,\delta }\text{ } \\
\text{and }B\otimes \alpha \circ \phi =\phi \circ \alpha \otimes B%
\end{array}%
\right\}
\end{eqnarray*}%
where the cocycle condition is the equality $\phi _{13}=\phi _{23}\phi
_{12}:N\otimes B\otimes B\rightarrow B\otimes B\otimes N.$ The subscripts
here indicate the factors that define $B\otimes B\rightarrow B\otimes
B\otimes B,$ i.e. $\phi _{12}\left( n\otimes b_{2}\otimes b_{3}\right) =\phi
\left( n\otimes b_{2}\right) \otimes b_{3},$ etc. There is a natural functor
$F:Mod_{A,\delta }\rightarrow Mod_{A\rightarrow B,\delta }$ given by
\begin{equation*}
F(M)=\left( M\otimes B,can:\left( M\otimes B\right) \otimes B\rightarrow
B\otimes \left( M\otimes B\right) \right) .
\end{equation*}%
Clearly the canonical map $can$ will satisfy the cocycle condition and maps
of $A$ modules will respect the canonical maps. The affine differential
descent theorem states that $F$ defines an equivalence of categories
(compare with \cite[Theorem 4.21]{MR2222646}).

\begin{theorem}
If $B$ is a faithfully flat differential extension of $A,$ then the functor $%
F:Mod_{A,\delta }\rightarrow Mod_{A\rightarrow B,\delta }$ is an equivalence
of categories with an inverse functor $G:Mod_{A\rightarrow B,\delta
}\rightarrow Mod_{A,\delta }$ defined by $G\left( \left( N,\phi \right)
\right) =\left\{ x\in N\mid \phi \left( 1\otimes x\right) =x\otimes
1\right\} .$
\end{theorem}

\begin{proof}
Descent theory for rings as exposed in \cite[Part 1, 4.2.1]{MR2222646}
establishes that $F$ and $G$ define equivalences of categories for modules.
If $A\rightarrow B$ is a differential ring homomorphism and $M$ is a
differential $A$ module, then $F\left( M\right) \in Mod_{A\rightarrow
B,\delta },$ and if $\left( N,\phi \right) \in Mod_{A\rightarrow B,\delta },$
then the fact that $\phi $ is a differential $B\otimes B$ module isomorphism
and an easy computation shows that $G\left( \left( N,\phi \right) \right) $
is closed under $\delta .$
\end{proof}

As an application of descent we can reprove and slightly extend Hochschild's
result.

\begin{corollary}
\label{diff}Let $\Lambda $ be an Azumaya $A$ algebra of rank $n^{2}$. Then
the derivation $\delta $ on $A$ can be extended to an $A$ derivation $\delta
:\Lambda \rightarrow \Lambda .$
\end{corollary}

\begin{proof}
There is an etale covering $U=Spec\left( B\right) \rightarrow X$ such that
there is an algebra isomorphism $f:M_{n}\left( B\right) \rightarrow \Lambda $
defining a $B\otimes B$ algebra automorphism $\phi :M_{n}\left( B\otimes
B\right) \rightarrow M_{n}\left( B\otimes B\right) $ such that $G\left(
\left( M_{n}\left( B\otimes B\right) ,\phi \right) \right) =\left\{ x\in
M_{n}\left( B\right) \mid \phi \left( 1\otimes x\right) =x\otimes 1\right\}
=\Lambda .$ Derivations extend uniquely to etale extensions and descent
theory tells us that $\phi $ is constructed as the composite
\begin{equation*}
\phi =1\otimes f\circ f\otimes 1:M_{n}\left( B\right) \otimes B\rightarrow
B\otimes M_{n}\left( B\right) .
\end{equation*}%
By transport of structure we can use $f$ to transfer coordinatewise
differentiation from $M_{n}\left( B\right) $ to a derivation of $\Lambda
\otimes B$ that extends $\delta $ on $B.to$ $\Lambda \otimes B.$ But this
makes $\phi $ a differential algebra isomorphism. Hence the coordinatewise
derivation on $M_{n}\left( B\right) $ restricts to a derivation on $G\left(
\left( M_{n}\left( B\otimes B\right) ,\phi \right) \right) $ that extends $%
\delta $ on $A$.
\end{proof}

If we fix a differential $A$ module $M$ and wish to describe isomorphism
classes of differential $A$ modules $M^{\prime }$ such that $M\otimes
B\approxeq _{\delta }M^{\prime }\otimes B,$ we are led, by the definition of
morphisms in the differential descent category, to consider descent data $%
\phi ^{\prime }$ for $M^{\prime }$ and $\phi ^{\prime \prime }$ for $%
M^{\prime \prime }$ and a $B$ module isomorphism $\alpha :M\otimes
B\rightarrow M\otimes B$ such that $\phi ^{\prime \prime }=\alpha _{2}\phi
\alpha _{1}^{-1}:\left( M\otimes B\right) \otimes B\rightarrow B\otimes
\left( M\otimes B\right) .$ Thus we define $Z^{1}\left( B/A,\underline{%
Aut_{\delta }\left( M\right) }\right) $ as%
\begin{equation*}
\left\{ \phi :M\otimes B\rightarrow B\otimes M\mid
\begin{array}{c}
\phi \text{ is a differential\ }B\otimes B\text{ module isomorphism } \\
\text{with }\phi _{13}=\phi _{23}\phi _{12}\text{ }%
\end{array}%
\right\}
\end{equation*}%
where $\underline{Aut_{\delta }\left( M\right) }$ is the differential
automorphism presheaf of $M.$ We define an equivalence relation on $%
Z^{1}\left( B/A,\underline{Aut_{\delta }\left( M\right) }\right) $ by $\phi
^{\prime }\thicksim \phi ^{\prime \prime }$ if there is an $\alpha $ as
above. Then we define
\begin{equation*}
H^{1}\left( B/A,\underline{Aut_{\delta }\left( M\right) }\right)
:=Z^{1}\left( B/A,\underline{Aut_{\delta }\left( M\right) }\right)
/\thicksim
\end{equation*}%
which leads to the key statement:%
\begin{equation*}
\left\{ M^{\prime }\mid M^{\prime }\otimes B\approxeq _{\delta }M\otimes
B\right\} =H^{1}\left( B/A,\underline{Aut_{\delta }\left( M\right) }\right) .
\end{equation*}

\begin{remark}
Note that the functor $G$ shows that if $M=\Lambda ,$ a differential $A$
algebra, and we require $\phi ^{\prime }$ to be a differential algebra
automorphism, then $\Lambda ^{\prime }=G\left( \Lambda \otimes B,\phi
^{\prime }\right) $ is also a differential $A$ algebra which becomes
isomorphic, as differential algebras, to $\Lambda $ over $B.$
\end{remark}

The differential descent theorem summarized above will be used to
characterize differential Azumaya algebras. But, in order to characterize
the differential Brauer group, we will need to construct the boundary map
for non-abelian cohomology. Since surjections of sheaves are not necessarily
surjections of presheaves, we will also need the definition and basic facts
about hypercoverings. We adopt the now standard simplicial notation but
refer to \cite{MR883959} for a more detailed discussion of hypercoverings.

Let $\mathbb{\Delta }$ denote the ordinal category whose objects are $\left[
0<1<\cdots <i\right] ,$ $0\leq i<\infty ,$ and morphisms are monotone maps,
and let $\mathbb{\Delta }_{n}$ denote the full subcategory whose objects are
$\left[ 0<1<\cdots <i\right] ,$ $0\leq i\leq n.$ Let $\mathcal{C}$ be a
category that is closed under difference kernels and finite products. Then $s%
\mathcal{C},$ the category of contravariant functors from $\mathbb{\Delta }$
to $\mathcal{C},$ denotes the category of simplicial objects in $\mathcal{C}$
and $s_{n}\mathcal{C},$ the category of contravariant functors from $\mathbb{%
\Delta }_{n}$ to $\mathcal{C},$ denotes the category of $n$ truncated
simplicial objects in $\mathcal{C}.$ The restriction functor $i_{n\ast }:s%
\mathcal{C}\rightarrow s_{n}\mathcal{C}$ which truncates a simplicial object
beginning at degree $n+1$ has a right adjoint $i_{n}^{!}:s_{n}\mathcal{C}%
\rightarrow s\mathcal{C}$ defined by%
\begin{equation*}
\left( i_{n}^{!}X\right) _{m}=\underleftarrow{\lim }_{k\rightarrow m}X_{k}
\end{equation*}%
where the limit is over all morphisms $k\rightarrow m$ in $\mathbb{\Delta }$
\ with $\left[ 0<1<\cdots <k\right] \rightarrow \left[ 0<1<\cdots <m\right] $
with $k\leq m.$ $Cosk_{n}\left( X\right) =i_{n}^{!}i_{n\ast }\left( X\right)
.$ Basically the coskeleton functor gives a way of canonically extending a
truncated simplicial object by introducing an $m$ simplex whenever
compatible $n$ faces exist in the $n$ truncated simplicial object. Thus, for
instance, if we assume $\mathcal{C}$ has a final object $X$ and we are given
a $0$ truncated simplicial object $U,$ then $Cosk_{0}\left( U\right)
_{m}=U\times _{X}\cdots \times _{X}U$ with factors indexed by vertices $%
0,\cdots ,m$ and face maps $d_{i}=p_{01\cdots \hat{\imath}\cdots
m}:Cosk_{0}\left( U\right) _{m}\rightarrow Cosk_{0}\left( U\right) _{m-1}$
where $p_{01\cdots \hat{\imath}\cdots m}$ indicates the projection map
omitting the $i^{th}$ factor.

Note that if $X$ is a topological space and $U$ is a disjoint union of open
sets that cover $X$ and $F$ is a sheaf of abelian groups on $X,$ then
applying $F$ to $Cosk_{0}\left( U\right) $ and setting $d^{m}=\sum_{i=0}^{m}%
\left( -1\right) ^{i}$ $F\left( d_{i}\right) $ produces the chain complex%
\begin{equation*}
F\left( U\right) \overset{d^{0}}{\rightarrow }F\left( U\times _{X}U\right)
\overset{d^{1}}{\rightarrow }F\left( U\times _{X}\cdots \times _{X}U\right)
\overset{d^{2}}{\rightarrow }\cdots \overset{d^{m-1}}{\rightarrow }F\left(
U\times _{X}\cdots \times _{X}U\right) \overset{d^{m}}{\rightarrow }\cdots
\end{equation*}%
whose homology, $\check{H}^{m}\left( U/X,F\right) ,$ is the Cech cohomology
of $F$ with respect to the covering.

Suppose $\mathcal{C}$ is a site with final object $X$, that is we suppose $%
\mathcal{C}$ has a Grothendieck topology on it. For simplicity, we assume
the topology is defined from a pretopology given by $Cov.$ Then
\begin{equation*}
\check{H}^{m}\left( X,F\right) :=\lim_{\underset{Cov\left( \mathcal{X}%
\right) }{\rightarrow }}\check{H}^{m}\left( U/X,F\right) .
\end{equation*}

\begin{definition}
A hypercovering $U_{\ast }$ of $X\in \mathcal{C}$ is a simplicial object
with values in $\mathcal{C}$ such that

\begin{itemize}
\item[(surj$_{0})$] $U_{0}\rightarrow X$ is a covering.

\item[(surj$_{n}$)] $U_{n+1}\rightarrow \left( Cosk_{n}U\right) _{n+1}$ is a
covering.
\end{itemize}
\end{definition}

We define $\check{H}^{m}\left( U_{\cdot },F\right) $ for a hypercovering in
a form similar to the above. Thus, for a presheaf $F,$ $\check{H}^{m}\left(
U_{\cdot },F\right) $ is the $m^{th}$ homology of the augmented complex%
\begin{equation*}
F\left( X\right) \rightarrow F\left( U_{0}\right) \overset{d^{0}}{%
\rightarrow }F\left( U_{1}\right) \overset{d^{1}}{\rightarrow }F\left(
U_{2}\right) \overset{d^{2}}{\rightarrow }\cdots \overset{d^{m-1}}{%
\rightarrow }F\left( U_{m}\right) \overset{d^{m}}{\rightarrow }\cdots
\end{equation*}%
where $d^{i}$ is the alternating sum of $F$ applied to the face maps of $%
U_{\cdot }.$ A map of hypercoverings $U_{\cdot }\rightarrow V_{\cdot }$ is a
simplicial map, and if $\mathbb{H}\left( X\right) $ denotes the category of
hypercoverings of $X$ in $\mathcal{C},$ $\mathbb{H}\left( X\right) $ is a
directed category. We define
\begin{equation*}
\check{H}^{m}\left( X_{\mathbb{H}},F\right) =\lim_{\mathbb{H}\left( X\right)
^{op}}\check{H}^{m}\left( U_{\cdot },F\right) .
\end{equation*}%
One of Verdier's main results about hypercoverings, \cite[Theorem 8.16]%
{MR883959}, then shows that this group agrees with sheaf cohomology if $F$
is an abelian sheaf. Note, however, that Cech cohomology, $\lim_{Cov\left(
X\right) }\check{H}^{m}\left( U_{.},F\right) ,$ agrees with $\lim_{\mathbb{H}%
\left( X\right) ^{op}}\check{H}^{m}\left( U_{.},F\right) $ for $m=0$ or $1.$

\begin{theorem}
Let $\mathcal{C}$ be a site that is closed under finite products, fibred
products and finite coproducts. Let $F$ be an abelian sheaf on $\mathcal{C}.$
Then there is a natural isomorphism $\check{H}^{m}\left( X,F\right)
\approxeq H^{m}\left( X,F\right) .$
\end{theorem}

Essentially the proof follows by observing that given an $n$ truncated
hypercovering $U_{\cdot }$ and a covering refinement $\phi :V\rightarrow
U_{n},$ then $i_{n}^{!}\left( V_{\cdot }\right) $ is a hypercovering where $%
V_{k}=U_{k}$ for $k<n$ and $V_{n}=V_{\cdot }$with face operators defined by
composing $\phi $ with the face operators of $U_{n}.$ We only need the case $%
n=1$ to define a boundary map for non-abelian cohomology. This will then
provide a cohomological interpretation of the differential Brauer group.

\begin{theorem}
\label{boundary}Let $\mathcal{C}$ be a site that is closed under finite
products, fibred products and finite coproducts with final object $X$. Let
\begin{equation*}
0\rightarrow C\overset{i}{\rightarrow }G\overset{j}{\rightarrow }%
H\rightarrow 1
\end{equation*}%
be a central extension of sheaves of groups for the site. Then there are
natural boundary maps $\partial ^{1}:H^{0}\left( X,H\right) \rightarrow
H^{1}\left( X,C\right) $ and $\partial ^{2}:\check{H}^{1}\left( X,H\right)
\rightarrow H^{2}\left( X,C\right) $ such that the sequence of groups and
pointed sets
\begin{eqnarray*}
0 &\rightarrow &H^{0}\left( X,C\right) \rightarrow H^{0}\left( X,G\right)
\rightarrow H^{0}\left( X,H\right) \overset{\partial ^{1}}{\rightarrow }%
H^{1}\left( X,C\right) \rightarrow \check{H}^{1}\left( X,G\right) \\
&\rightarrow &\check{H}^{1}\left( X,H\right) \overset{\partial ^{2}}{%
\rightarrow }H^{2}\left( X,C\right)
\end{eqnarray*}%
is exact.
\end{theorem}

\begin{proof}
This is a straightforward computation except for the definition of $\partial
^{2}$ and the exactness of pointed sets at $H^{1}\left( X,H\right) .$
Suppose $x\in \check{H}^{1}\left( X_{\cdot },G\right) $ is represented by a
cocycle $c_{U}\in H\left( U\mathbb{\times }_{X}U\right) $ for a cover $%
U\rightarrow X.$ Then there is a covering refinement $V\rightarrow U\mathbb{%
\times }_{X}U$ generating $i_{1}^{!}V$ as above and a lift of $c_{U}\mid
_{i_{1}^{!}V_{1}}$ to $\ell \left( c_{U}\right) \in G\left(
i_{1}^{!}V_{1}\right) $ such that $j\left( \ell \left( c_{U}\right) \right)
=c_{U},$ and we define $\partial \left( c_{U}\right) $ by
\begin{equation*}
i\left( \partial \left( c_{U}\right) \right) :=\left( \ell \left(
c_{U}\right) \right) _{13}^{-1}\left( \ell \left( c_{U}\right) \right)
_{23}\left( \ell \left( c_{U}\right) \right) _{12}\in Ker\left[ \Gamma
\left( i_{1}^{!}V_{2},G\right) \rightarrow \Gamma \left(
i_{1}^{!}V_{2},H\right) \right] .
\end{equation*}%
Here the subscripts indicate that the face operator of the simplicial
covering used is the missing index. Altering the lift of $c_{U}$ alters $%
\partial \left( c_{U}\right) \in \Gamma \left( i_{1}^{!}V_{2},C\right) $ by
a coboundary since $C$ is central in $G.$ Replacing $c_{U}$ with a
cohomologous cocycle, $c_{U}^{^{\prime }}=\alpha _{2}c_{U}\alpha _{1}^{-1},$
is dealt with by refining the Cech covering until we can replace $\ell
\left( c_{U}\right) $ with a cohomologous lift using $\beta \in \Gamma
\left( U_{0}^{^{\prime }},G\right) $\ where $\beta $ maps to the restriction
of $\alpha $ to $U_{0}^{^{\prime }}.$ Then observe that
\begin{eqnarray*}
&&\left( \beta _{2}\ell \left( c_{U}\right) \beta _{1}^{-1}\right)
_{13}^{-1}\left( \beta _{2}\ell \left( c_{U}\right) \beta _{1}^{-1}\right)
_{23}\left( \beta _{2}\ell \left( c_{U}\right) \beta _{1}^{-1}\right) _{12}
\\
&=&\left( \beta _{1}\right) _{13}\left( \ell \left( c_{U}\right) \right)
_{13}^{-1}\left( \ell \left( c_{U}\right) \right) _{23}\left( \ell \left(
c_{U}\right) \right) _{12}\left( \beta _{1}^{-1}\right) _{12}
\end{eqnarray*}%
since, in our notation showing the factors of the face maps,
\begin{equation*}
\left( \beta _{2}\right) _{13}^{-1}\left( \beta _{2}\right) _{23}=\left(
\beta _{1}\right) _{23}^{-1}\left( \beta _{2}\right) _{12}=1.
\end{equation*}
But $i\left( \partial \left( c_{U}\right) \right) $ is in the center of $%
\Gamma \left( i_{1}^{!}V_{2},G\right) $ and so the final pair $\left( \beta
_{1}\right) _{13}\left( \beta _{1}^{-1}\right) _{12}=1.$ A similar
computation shows that $\partial \left( c_{U}\right) \in Z^{2}\left(
X,i_{1}^{!}V_{\cdot }\right) $ and so defines an element $\left[ \partial
\left( c_{U}\right) \right] \in H^{2}\left( X,C\right) .$ Finally exactness
follows easily since $\left[ \partial \left( c_{U}\right) \right] =0$ says
that for some hypercovering $U_{\cdot }^{\prime }$ of $X$ (which we may
assume maps to $i_{1}^{!}V_{2}$), the restriction to $U_{2}^{\prime }$ of $%
\partial \left( c_{U}\right) =d^{1}\left( u\right) $ for some $u\in $ $%
C\left( U_{1}^{\prime }\right) .$ But then multiplying the restriction of $%
\ell \left( c_{U}\right) $ by $i\left( u^{-1}\right) $ produces a cocycle in
$Z^{1}\left( U_{1}^{^{\prime }},G\right) $ that maps to the restriction of $%
c_{U}$ to $\Gamma \left( U_{1}^{^{\prime }},H\right) .$ The other part is
clear.
\end{proof}

\section{Cohomological Interpretation}

We introduce the $\delta -flat$ topology in order to apply this material in
a differential setting. Picard-Vessiot extensions were the key tool used in
\cite{MR2383496} to analyze automorphisms of differential central simple
algebras. The no new constant requirement makes constructing such extensions
in the differential ring setting difficult. However what was really used in
their work was the existence of solutions to certain linear differential
equations and that we can achieve locally with the $\delta -flat$ topology
on a differential ring. In order to match the notation in Milne \cite%
{MR559531}, our basic reference for this material, we let $X=Spec\left(
A\right) $ be a differential affine scheme with respect to a vector field
denoted by $\delta $ for the remainder of this note.

Let $\mathcal{C}_{\delta }\left( X\right) $ denote the category of
differential affine schemes over $X$ and differential maps. The $\delta $%
-flat topology on $X$ is the topology defined by the class of morphisms
\begin{equation*}
E=\text{all flat differential morphisms that are locally of finite type.}
\end{equation*}%
Thus if $U\in \mathcal{C}_{\delta }\left( X\right) ,$ a covering of $U$ is a
finite family$\left\{ U_{i}\overset{g_{i}}{\rightarrow }U\right\} _{i\in I}$
such that $U_{i}$ is flat and locally of finite type over $U,$ $g_{i}$ is a
differentiable map, and $U=\cup _{i\in I}g_{i}\left( U_{i}\right) .$ We can
and will replace $\left\{ U_{i}\overset{g_{i}}{\rightarrow }U\right\} _{i\in
I}$ with a single map $U^{\prime }:=\bigsqcup\limits_{i\in
I}U_{i}\rightarrow U$ to match the notation above. Technically we have
defined a pre-topology and should not require the covering family to be
finite.. But this defines a unique topology on $X$ and is more convenient to
work with. We should also note that this pretopology should have been
defined on the category of all differential schemes over $X,$ not just
affine ones. However we have lost nothing by restricting to affine covers as
the argument in \cite[Chapter II, Proposition 1.5]{MR559531} shows. We
denote this site by $X_{\delta -fl}.$

Examples of sheaves for this topology are easy to find since coverings in
the $\delta $-flat topology are coverings in the flat topology. Consequently
any sheaf for the flat topology restricts to a sheaf on $X_{\delta -fl}.$ In
particular any $A$ module $M$ defines a sheaf on $X_{\delta -fl}$ which we
denote $W\left( M\right) $ as in \cite[II, Example 1.2 d)]{MR559531} by
setting $\Gamma \left( U,W\left( M\right) \right) =M\otimes B$ if $%
U=Spec\left( B\right) .$ An alternative way of constructing $\delta -$flat
sheaves is to take the subsheaf of differential constants of a sheaf of
groups $F$ with a $\delta $ action since kernels preserve equalizers. Here
are some examples and the notation we will use.

\begin{itemize}
\item $\mathbb{G}_{m,\delta }:=Aut_{\delta }W\left( A\right) $ where $\left(
A,\delta \right) $ is regarded as a free $A$ module with coordinatewise
differentiation. Note that differential automorphisms of the free rank one
module with a constant differential basis are given by multiplication by a
differentially constant unit. Hence $\mathbb{G}_{m,\delta }=\mathbb{G}%
_{m}^{\delta }.$

\item $\mathbb{G}l_{n,\delta }:=Aut_{\delta }\left( W\left(\bigoplus\limits_{i=1}^{n}Ae_{i}\right) \right) $ where $%
\bigoplus\limits_{i=1}^{n}Ae_{i}$ is a free module with coordinatewise
differentiation, i.e. $\delta e_{i}=0$ for all $i.$ We indicate this module
as $\left( \bigoplus\limits_{i=1}^{n}A,^{\prime }\right) .$ Note that
differential automorphisms of this module with differential basis $\left\{
e_{i}\right\} $ are given by multiplication by a differentially constant
invertible matrix of rank $n$ since each constant basis vector must be sent
to a differentially constant column vector. Hence $\mathbb{G}l_{n,\delta }=%
\mathbb{G}l_{n}^{\delta }.$

\item $P\mathbb{G}l_{n,\delta }:=Aut_{\delta }\left( \left( M_{n},^{\prime
}\right) \right) $ where $\left( M_{n},^{\prime }\right) $ denotes the sheaf
of $n\times n$ matrices with coordinatewise differentiation.
\end{itemize}

These automorphism sheaves allow us to give a cohomological interpretation
of $Az_{n,\delta }$ and $\Pr oj_{n,\delta }.$ For this we will need the
following lemma.

\begin{proposition}
Let $\left( A,\delta \right) $ be a differential ring. Then

\begin{enumerate}
\item $H^{1}\left( A_{\delta -fl},\mathbb{G}_{m,\delta }\right) =Pic_{\delta
}\left( A\right) $,

\item $H^{1}\left( A_{\delta -fl},\mathbb{G}l_{n,\delta }\right) =\Pr
oj_{n,\delta }\left( A\right) ,$

\item $H^{1}\left( A_{\delta -fl},P\mathbb{G}l_{n,\delta }:\right)
=Az_{n,\delta }\left( A\right) .$
\end{enumerate}
\end{proposition}

\begin{proof}
The result follows from $\delta $-flat descent theory once we have shown in
each claim that any differential object of the correct rank can be made
differentially isomorphic to the model by a covering in the $\delta $-flat
topology. The first claim is a special case of the second to which we now
turn. Let $P$ be a finitely generated, projective differential sheaf of rank
$n.$ Then there is a Zariski covering $U=\bigsqcup\limits_{i\in
I}U_{i}\rightarrow X$ such that $P\mid _{U}\approxeq \oplus _{1}^{n}\mathcal{%
O}_{U}^{n}e_{i}$ where $U=Spec\left( A^{\prime }\right) $ is a differential
affine scheme over $X.$ Fix a basis $\left\{ e_{1},\ldots ,e_{n}\right\} $
for $P,$ and let
\begin{equation*}
E=\left(
\begin{array}{ccc}
e_{1} & \cdots  & 0 \\
\vdots  & \ddots  & \vdots  \\
0 & \cdots  & e_{n}%
\end{array}%
\right)
\end{equation*}%
be an $n\times n$ diagonal matrix whose columns represent this basis. (We
write $e_{1},\ldots ,e_{n}$ instead of the more usual $1,\ldots ,1$ since $%
\delta e_{j}$ may be non-zero.). Then $\left( \delta E\right) E$ has $j^{th}$
column $\delta e_{j}$ expanded in terms of the basis $\left\{ e_{1},\ldots
,e_{n}\right\} .$ Let $t=\left( t_{ij}\right) $ be a matrix of variables so
that the columns of%
\begin{equation*}
tE
\end{equation*}%
represent another basis written in terms of $E$ if $t$ is invertible. Then
\begin{equation*}
\delta \left( tE\right) =\delta \left( t\right) E+t\delta \left( E\right) E=%
\overrightarrow{0}
\end{equation*}%
is the system of linear differential equations that we must solve. We let $%
V=Spec\left( A^{\prime \prime }\right) $ where $A^{\prime \prime }=A^{\prime
}\left[ T_{ij,}\det \left( t_{ij}\right) ^{-1}\right] $ is the differential
ring in $n^{2}$ indeterminates $T_{ij}$ with derivation $\delta $ defined by
$\left( \delta T_{ij}\right) =-\left( T_{ij}\right) \left( \delta E\right) .$
Then $A^{\prime }\rightarrow A^{\prime \prime }$ is a differential ring
homomorphism and $P\otimes _{A^{\prime }}A^{\prime \prime }$ has a basis
over $A^{\prime \prime }$ consisting of the column vectors in $\left(
T_{i,j}\right) E,$ which is a basis of differential constants.

Finally if $\Lambda \in Az_{n,\delta }\left( A\right) ,$ there is an \'{e}%
tale covering $U=\bigsqcup\limits_{i\in I}U_{i}\rightarrow X$ where $%
U=Spec\left( A^{\prime }\right) $ such that $\Lambda \otimes A^{\prime
}\approxeq M_{n}\left( A^{\prime }\right) .$ Since derivations extend
uniquely to \'{e}tale coverings, the covering map is a map of differential
schemes. This reduces the assertion to showing that, locally for the $\delta
$-flat topology, the differential Azumaya algebra, $\left( M_{n}\left(
A^{\prime }\right) ,\delta \right) $ is isomorphic as a differential algebra
to $\left( M_{n}\left( A^{\prime }\right) ,^{\prime }\right) .$ But if $%
\delta $ is not coordinate wise differentiation, then $\delta -^{\prime }$
is a derivation of $M_{n}\left( A^{\prime }\right) $ that extends the zero
derivation on $A^{\prime }.$ Hence, by the Corollary \ref{der}, there is a $%
z\in M_{n}\left( A^{\prime }\right) $ such that $\delta =^{\prime }+\left[
z,-\right] .$ In order to make the $z$ unique we require that its trace be $%
0.$ We must show that there is a differential algebra isomorphism $\phi
:\left( M_{n}\left( A^{\prime }\right) ,\delta \right) \rightarrow \left(
M_{n}\left( A^{\prime }\right) ,^{\prime }\right) .$ Locally any such
algebra sutomorphism, $\phi _{u},$ is given by automorphism with a unit $%
u\in M_{n}\left( A^{\prime }\right) .$ But $\phi _{u}$ is a differential
isomorphism if and only if
\begin{equation*}
0=\left( uxu^{-1}\right) ^{\prime }-u\left( \delta x\right) u^{-1}=u^{\prime
}xu^{-1}-uxu^{-1}u^{\prime }u^{-1}-u\left( zx-xz\right) u^{-1},
\end{equation*}%
and this will hold if and only if $u^{-1}u^{\prime }-z\in A^{\prime }$ where
$z=\left( z_{ij}\right) \in M_{n}\left( A^{\prime }\right) .$ (Note that the
choice of $u$ defining $\phi _{u}$ is not unique, but it is unique up to a
unit from $A^{\prime }$ and this alters $u^{-1}u^{\prime }$ by an element of
$A^{\prime }.)$ So we let $A^{\prime \prime }=A^{\prime }[U_{ij},\det \left(
\left( U_{ij}\right) \right) ^{-1}]$ where $U_{ij}$ are indeterminates with $%
\delta $ extended to $A^{\prime \prime }$ as $\left( U_{ij}\right) ^{\prime
}=\left( z_{ij}\right) \left( U_{ij}\right) .$Then there is a differential
algebra isomorphism $\phi _{u}$ which defines an isomorphism over $A^{\prime
\prime }$ between $\left( M_{n}\left( A^{\prime ^{\prime }}\right) ,\delta
\right) $ and $\left( M_{n}\left( A^{\prime ^{\prime }}\right) ,^{\prime
}\right) .$
\end{proof}

The central extension
\begin{equation*}
0\rightarrow \mathbb{G}_{m}\rightarrow \mathbb{G}l_{n}\rightarrow P\mathbb{G}%
l_{n}\rightarrow 1,
\end{equation*}%
is used to construct the embedding $Br\left( A\right) \subset H^{2}\left(
A_{fl},\mathbb{G}_{m}\right) .$ (We use a $0$ if the corresponding group is
abelian; otherwise we use a $1$ as above.) We establish the corresponding
exact sequence of sheaves on $X_{\delta -fl}$.

\begin{theorem}
Let $X=Spec\left( A\right) $ be an affine scheme with a derivation $\delta $
on $A.$

\begin{enumerate}
\item There is an exact sequence of sheaves of abelian groups on $X_{d-fl}:$%
\begin{equation*}
0\rightarrow \mathbb{G}_{m,\delta }\rightarrow \mathbb{G}_{m}\overset{d\ln }{%
\rightarrow }W\left( A\right) \rightarrow 0
\end{equation*}

\item There is a central extension of sheaves of groups on $X_{\delta -fl}:$%
\begin{equation*}
0\rightarrow \mathbb{G}_{m,\delta }\rightarrow \mathbb{G}l_{n,\delta
}\rightarrow P\mathbb{G}l_{n,\delta }\rightarrow 1
\end{equation*}
\end{enumerate}
\end{theorem}

\begin{proof}
1): The map $d\ln $ is defined by $d\ln \left( u\right) =\left( \delta
u\right) /u.$ for $u\in \Gamma \left( Spec\left( B\right) ,\mathbb{G}%
_{m}\right) .$ The logarithmic derivative map is additive and has kernel $%
\mathbb{G}_{m,\delta },$ and so the only question is surjectivity. If $b\in
\Gamma \left( Spec\left( B\right) ,W\left( A\right) \right) =B,$ then the
linear differential equation $\left( \delta x\right) =bx$ has a solution in
the differential extension $B\left[ U\right] $ where $U$ is an indeterminate
and differentiation is defined by $\delta U=bU.$ The differentiation extends
to define a differential $B$ algebra $B_{1}:=B\left[ U,U^{-1}\right] $ which
will be a $\delta -$flat cover of $B$ and $d\ln \left( U\right) =b$ in $%
B_{1}.$ Consequently the homomorphism of sheaves$\ d\ln $ is surjective
locally.

2): By definition $P\mathbb{G}l_{n,\delta }$ is the $\delta $-flat
automorphism sheaf of $\left( M_{n},^{\prime }\right) .$ Thus it suffices to
show that for any $U$ and $\phi \in $ $\Gamma \left( U,\underline{%
Aut_{\delta }\left( \left( M_{n},^{\prime }\right) \right) }\right) ,$ there
is a $\delta $-flat covering $U^{\prime }\rightarrow U$ and $u\in \Gamma
\left( U^{\prime },\mathbb{G}l_{n,\delta }\right) $ such that $\phi \mid
_{U^{\prime }}$ is given by conjugation by $u.$ The Differential Skolem
Noether Theorem states that this will happen if $L=Hom_{B}\left( \left(
\bigoplus\limits_{i=1}^{n}A_{U^{\prime }}e_{i},^{\prime }\right) ,\left(
\bigoplus\limits_{i=1}^{n}A_{U^{\prime }}e_{i},^{\prime }\right) _{\phi
}\right) _{\delta }$ is differentially isomorphic to $\left( A_{U^{\prime
}},^{\prime }\right) $ on $U^{\prime }$ where, for simplicity, we have
written $A_{U^{\prime }}$ for $\Gamma \left( U^{\prime },\mathcal{O}%
_{U^{\prime }}\right) .$ First choose a Zariski covering $U^{\prime
}\rightarrow U$ which makes $L\approxeq A_{U^{\prime }}$ on $U^{\prime }$
but not necessarily differentially isomorphic. Hence $\phi \left( x\right)
=uxu^{-1}$ for some $u\in \mathbb{G}l_{n}\left( A_{U^{\prime }}\right) $ and
any $x\in M_{n}\left( A_{U^{\prime }}\right) .$ Now, by Lemma \ref{mapd}, we
must have
\begin{equation*}
u^{-1}\left( \delta u\right) =a
\end{equation*}%
for some $a\in A_{U^{\prime }}.$ Finally find a covering $U^{\prime \prime
}\rightarrow U$ such that the differential equation $\delta X=-aX$ has, as a
solution, a unit $b\in A_{U^{\prime \prime }}.$ Then, on $U^{\prime \prime },
$ $\phi \left( x\right) =\left( bu\right) x\left( bu\right) ^{-1}$ and $%
\delta \left( bu\right) =0.$ Consequently $bu\in Gl_{n,\delta }\left(
A_{U^{\prime \prime }}\right) .$
\end{proof}

Now we are in a position to put everything together. We begin by
constructing an embedding of $Br_{\delta }\left( X\right) \hookrightarrow
H^{2}\left( X_{\delta -pl},\mathbb{G}_{m}\right) .$

\begin{proposition}
Let $A$ be a connected, commutative differential ring. Let $\partial
_{n}^{2}:Az_{n,\delta }\left( A\right) \rightarrow H^{2}\left( A_{\delta
-pl},\mathbb{G}_{m,\delta }\right) $ be the boundary map of Theorem \ref%
{boundary}. If $\Lambda \in Az_{m,\delta }$ and $\Gamma \in Az_{n,\delta },$
then
\begin{equation*}
\partial _{mn}^{2}\left( \Lambda \otimes \Gamma \right) =\partial
_{m}^{2}\left( \Lambda \right) +\partial _{n}^{2}\left( \Gamma \right) \in
H^{2}\left( A_{\delta -pl},\mathbb{G}_{m,\delta }\right)
\end{equation*}%
and $\partial _{m}^{2}\left( \Lambda ^{op}\right) =-$ $\partial
_{m}^{2}\left( \Lambda \right) .$
\end{proposition}

\begin{proof}
If $\Lambda $ and $\Gamma $ are represented by $1$ cocycles $c_{\Lambda }\in
\Gamma \left( U^{\prime },P\mathbb{G}l_{m}\right) $ and $c_{\Gamma }\in
\Gamma \left( U^{\prime \prime },P\mathbb{G}l_{n}\right) $ respectively,
then, choosing a cover $U$ such that $U\rightarrow U^{\prime }$ and $%
U\rightarrow U^{\prime \prime },$  $\Lambda \otimes \Gamma $ is represented
by $c_{\Lambda }\mid \otimes c_{\Gamma }\mid \in \Gamma \left( U,P\mathbb{G}%
l_{mn}\right) $ and so, in the notation of Theorem \ref{boundary}, $\ell
\left( c_{\Lambda ,U}\right) \otimes \ell \left( c_{\Gamma ,U}\right) $ is a
lift of $c_{\Lambda \otimes \Gamma ,U}\in \Gamma \left( U\times U,P\mathbb{G}%
l_{mn}\right) .$ Since the tensor factors in $\mathbb{G}l_{m}\otimes \mathbb{%
G}l_{n}=\mathbb{G}l_{mn}$ commute, $\partial _{mn}^{2}\left( \Lambda \otimes
\Gamma \right) =\partial _{m}^{2}\left( \Lambda \right) +\partial
_{n}^{2}\left( \Gamma \right) .$ Since $\Lambda \otimes \Lambda
^{op}\approxeq End_{A,\delta }\left( \Lambda \right) ,$ $\partial
_{m^{2}}^{2}\left( \Lambda \otimes \Lambda ^{op}\right) =0.$
\end{proof}

Thus the exact sequence of pointed sets in Theorem \ref{boundary} together
with this Proposition allow us to define a homomorphism
\begin{equation*}
\partial ^{2}:\underset{\rightarrow }{\lim }_{M\mid N}Az_{M,\delta }\left(
X\right) \rightarrow H^{2}\left( X_{\delta -pl},\mathbb{G}_{m,\delta
}\right)
\end{equation*}%
whose kernel is $\underset{\rightarrow }{\lim }_{M\mid N}\Pr oj_{M,\delta
}\left( X\right) $ since $P$ corresponds to $End\left( P\right) .$ This
defines a one-to-one group homomorphism $\partial :Br_{\delta }\left(
X\right) \rightarrow H^{2}\left( X_{\delta -pl},\mathbb{G}_{m}\right) $ and
so $Br_{\delta }\left( X\right) \subseteq H^{2}\left( X_{\delta -pl},\mathbb{%
G}_{m}\right) .$ \ We are now in a position to reap the benefits of our work.

\begin{theorem}
Let $X=Spec\left( A\right) $ where $A$ is a connected, commutative
differential ring. Then $Br_{\delta }\left( X\right) $ is torsion and $%
\partial :Br_{\delta }\left( X\right) \rightarrow \,_{tors}H^{2}\left(
X_{\delta -pl},\mathbb{G}_{m}\right) $ is an isomorphism. The map that
forgets the differentiation induces an isomorphism $Br_{\delta }\left(
A\right) \rightarrow Br\left( A\right) .$
\end{theorem}

\begin{proof}
The long exact sequence of cohomology associated to the short exact sequence
of coefficient sheaves $0\rightarrow \mathbb{G}_{m,\delta }\rightarrow
\mathbb{G}_{m}\overset{d\ln }{\rightarrow }W\left( A\right) \rightarrow 0$
is, in part
\begin{equation*}
\begin{array}{cccc}
H^{1}\left( X_{\delta -pl},W\left( A\right) \right) \rightarrow  &
H^{2}\left( X_{\delta -pl},\mathbb{G}_{m,\delta }\right) \rightarrow  &
H^{2}\left( X_{\delta -pl},\mathbb{G}_{m}\right) \rightarrow  & H^{2}\left(
X_{\delta -pl},W\left( A\right) \right)  \\
& \cup  & \cup  &  \\
& Br_{\delta }\left( X\right) \rightarrow  & Br\left( X\right)  &
\end{array}%
\end{equation*}%
where the vertical inclusions come from the boundary map. There is a
morphism of sites given by the identity map on $X,$ $f:X_{pl}\rightarrow
X_{\delta -pl}$ since any $\delta $-flat covering is a flat covering, and we
have $H^{i}\left( X_{\delta -pl},f_{\ast }\mathcal{F}^{\prime }\right)
\rightarrow H^{i}\left( X_{pl},\mathcal{F}^{\prime }\right) $ is an
isomorphism for any sheaf $\mathcal{F}^{\prime }$ on $X_{pl}.$ ( See \cite[%
III, Proposition 3.1]{MR559531}). This shows that $H^{2}\left( X_{\delta
-pl},\mathbb{G}_{m}\right) \approxeq H^{2}\left( X_{pl},\mathbb{G}%
_{m}\right) $ since $f_{\ast }\mathbb{G}_{m}$ is the sheaf of units on $%
X_{\delta -pl}$ and similarly that $H^{1}\left( X_{\delta -pl},W\left(
A\right) \right) \approxeq H^{1}\left( X_{pl},W\left( A\right) \right) .$
But (\cite[III, Proposition 3.7]{MR559531}) $H^{1}\left( X_{pl},W\left(
A\right) \right) \approxeq H^{1}\left( X_{Zar},W\left( A\right) \right) =0$
for $i>0$ since $X$ is affine. It is well known that $Br\left( X\right) $ is
isomorphic to the torsion in $H^{2}\left( X_{pl},\mathbb{G}_{m}\right) .$
Finally if $\Lambda $ is an Azumaya algebra on $X,$ then $\delta $ can be
extended to a derivation on $\Lambda $ by Corollary \ref{diff} and so the
bottom map in the diagram is surjective as well as injective.
\end{proof}

The key to this result was the vanishing of coherent sheaf cohomology on
affine varieties. While everything has been constructed for this case, it is
clear that all of the basics extend to a variety with a vector field. As an
example consider an abelian variety $X$ of dimension $g$ over a
characteristic zero field $k$. Then $\Omega
_{X/k}^{1}=\bigoplus\limits_{i=1}^{g}\mathcal{O}_{X}$ and so we can define
a global vector field $v$ on $X$ by choosing a global homomorphism $\phi
_{V}:\Omega _{X/k}^{1}\rightarrow \mathcal{O}_{X}.$ In this case, just as in
the theorem, $H^{i}\left( X_{\delta -pl},\mathbb{G}_{m}\right) =H^{i}\left(
X_{pl},\mathbb{G}_{m}\right) $ for all $i$ and $H^{2}\left( X_{\delta -pl},%
\mathbb{G}_{m}\right) \approxeq Br\left( X\right) $ since $X$ is smooth.
It's easy to see that $Br_{\delta }\left( X\right) $ is torsion (since $%
\partial $ factors through $H^{2}\left( X_{\delta -pl},\mu _{\infty }\right)
$)$.$ The long exact cohomology sequence then becomes%
\begin{equation*}
\cdots \rightarrow H^{1}\left( X_{\delta -pl},\mathbb{G}_{m}\right)
\rightarrow Pic\left( X\right) \rightarrow H^{1}\left( X,\mathcal{O}%
_{X}\right) \rightarrow H^{2}\left( X_{\delta -pl},\mathbb{G}_{m}\right)
\rightarrow Br\left( X\right) \rightarrow 0.
\end{equation*}%
The last map is surjective since $H^{2}\left( X,\mathcal{O}_{X}\right) $ is
a $k$ vector space and so torsion free but $Pic\left( X\right) \rightarrow
H^{1}\left( X,\mathcal{O}_{X}\right) \approxeq \oplus _{i=1}^{g}k$ can't be
surjective and, generally this will contribute a torsion free part to $%
H^{2}\left( X_{\delta -pl},\mathbb{G}_{m}\right) .$ We plan on discussing
the non-affine case in more detail in a future paper.

\bibliographystyle{amsplain}
\bibliography{HooblerBibTex}

\end{document}